\documentclass[12pt,reqno]{amsart}
\usepackage{enumerate, latexsym, amsmath, amsfonts, amssymb, amsthm, color}
\def\pmod #1{\ ({\rm{mod}}\ #1)}

\def\l{\left}
\def\r{\right}
\def\bg{\bigg}
\def\({\bg(}
\def\){\bg)}

\theoremstyle{plain}
\newtheorem{theorem}{Theorem}

\newtheorem{lemma}{Lemma}

\theoremstyle{definition}

\theoremstyle{remark}

\allowdisplaybreaks

\begin{document}
\title
[{Supercongruences involving Domb numbers}]
{Supercongruences involving Domb numbers and binary quadratic forms}







\author{Guo-Shuai Mao}
\address{Department of Mathematics, Nanjing
University of Information Science and Technology, Nanjing 210044, 
People's Republic of China}
\email{maogsmath@163.com}
\thanks{The first author was supported by
  the Natural Science Foundation of China (grant 12001288)
  and the China Scholarship Council (202008320187).}

\author{Michael J.\ Schlosser}
\address{Fakult\"at f\"ur Mathematik, Universit\"at Wien,
Oskar-Morgenstern-Platz~1, A-1090 Vienna, Austria}
\email{michael.schlosser@univie.ac.at}
\thanks{The second author was partially supported by FWF Austrian Science
Fund grant P 32305.}

\begin{abstract}
    In this paper, we prove two recently conjectured supercongruences
(modulo $p^3$, where $p$ is any prime greater than $3$)
of Zhi-Hong Sun on truncated sums involving the Domb numbers.
Our proofs involve a number of ingredients such as congruences involving
specialized Bernoulli polynomials, harmonic numbers, binomial coefficients,
and hypergeometric summations and transformations.
\end{abstract}


\maketitle

\section{Introduction}
\setcounter{lemma}{0}
\setcounter{theorem}{0}
\setcounter{corollary}{0}
\setcounter{remark}{0}
\setcounter{equation}{0}
\setcounter{conjecture}{0}

The \textit{Domb numbers} $\{D_n\}$, defined by
$$D_n=\sum_{k=0}^n\binom{n}k^2\binom{2k}{k}\binom{2n-2k}{n-k}$$
for non-negative integers $n$,
first appeared in an extensive study by C.~Domb~\cite{D}
on interacting particles on crystal lattices. In particular, Domb showed that
$D_n$ counts the number of $2n$-step polygons on the diamond lattice.

The Domb numbes also appear in a variety of other settings, such
as in the coefficients in several known series for $1/\pi$.
For example, from \cite[Equation~(1.3)]{CCL} we know that
$$
\sum_{n=0}^\infty\frac{5n+1}{64^n}D_n=\frac{8}{\sqrt3\pi}.
$$

In \cite[Theorem~3.1]{R}, M.D.~Rogers showed the following generating
function for the Domb numbers by applying a rather intricate method:
$$
\sum_{n=0}^\infty D_nu^n=
\frac1{1-4u}\sum_{k=0}^\infty\binom{2k}k^2
\binom{3k}k\l(\frac{u^2}{(1-4u)^3}\r)^k,
$$
where $|u|$ is sufficiently small.
Y.-P. Mu and Z.-W. Sun \cite[Equation~(1.11)]{MS18} proved a
congruence involving the Domb numbers by applying
the telescoping method:
For any prime $p>3$, we have the supercongruence
$$
\sum_{k=0}^{p-1}\frac{3k^2+k}{16^k}D_k\equiv-4p^4q_p(2)\quad\pmod {p^5},
$$
where $q_p(a)$ denotes the Fermat quotient $(a^{p-1}-1)/p$.

In \cite{liud}, J.-C. Liu proved a couple of conjectures of Z.-W. Sun
and Z.-H. Sun. In particular he confirmed \cite[Theorem~1.3]{liud} that
for any positive integer $n$ the two sums
$$
\frac1n\sum_{k=0}^{n-1}(2k+1)D_k8^{n-1-k}\ \ \ \mbox{and}\ \ \ \frac1n\sum_{k=0}^{n-1}(2k+1)D_k(-8)^{n-1-k}
$$
are also positive integers.

Z.-H. Sun \cite[Conjecture~4.1]{SH2} conjectured the following congruence
for the Domb numbers: Let $p>3$ be a prime. Then
$$
D_{p-1}\equiv64^{p-1}-\frac{p^3}6B_{p-3}\quad\pmod{p^4},
$$
where $\{B_n\}$  are the Bernoulli numbers given by
$$
 B_0=1,\ \ \ \sum_{k=0}^{n-1}\binom{n}{k}B_{k}=0\ \ (n\geq2).
$$
This conjecture was confirmed by the first author and
J.~Wang \cite{mw}.
For more research on Domb numbers, we kindly refer the readers
to \cite{liud,mp3,sund,S2020,S14} (and the references therein).

The main result of this paper is Theorem~\ref{Th1.1}
which contains two supercongruences that were originally conjectured
by Z.-H. Sun in \cite[Conjecture 3.5, Conjecture 3.6]{s2111}.
What makes them interesting is that their formulations involve
the binary quadratic form $x^2+3y^2$ for primes $p$ that are congruent to $1$
modulo $3$. (It is well-known that any prime $p\equiv 1\pmod 3$
can be expressed as $p=x^2+3y^2$ for some integers $x$ and $y$, an assertion
first made by Fermat and subsequently proved by Euler, see \cite{cox}.
In his paper~\cite{s2111}, Sun stated further conjectures of similar type,
involving different moduli, and other binary quadratic forms.)
First, Sun defined
$$
R_3(p)=\left(1+2p+\frac43(2^{p-1}-1)
  -\frac32(3^{p-1}-1)\right)\binom{\frac{p-1}2}{\lfloor p/6\rfloor}^2.
$$
The two supercongruences which we will confirm are as follows.
\begin{theorem}\label{Th1.1} Let $p>3$ be a prime. Then
 \begin{align*}
   &\sum_{k=0}^{p-1}k^3\frac{D_k}{4^k}\\
   &\equiv\begin{cases}-\frac{64}{45}x^2+\frac{32}{45}p+\frac{43p^2}{90x^2}\;\pmod{p^3} &\tt{if}\ \textit{p}=\textit{x}^2+3\textit{y}^2\equiv1\pmod 3,
\\ \frac{28}9R_3(p)\;\pmod{p^2} &\tt{if}\ \textit{p}\equiv2\pmod3\ \mbox{and}\ \textit{p}\neq5,\end{cases}\\
   &\sum_{k=0}^{p-1}k^3\frac{D_k}{16^k}\\
   &\equiv\begin{cases}\frac{4}{45}x^2-\frac{2}{45}p+\frac{p^2}{45x^2}\;\pmod{p^3} &\tt{if}\ \textit{p}=\textit{x}^2+3\textit{y}^2\equiv1\pmod 3,
\\ -\frac{4}9R_3(p)\;\pmod{p^2} &\tt{if}\ \textit{p}\equiv2\pmod3.\end{cases}
\end{align*}
\end{theorem}
Our preparations for the proof of this theorem consist of seven lemmas that we
give in Section~\ref{sec:pre}. These are used in Section~\ref{sec:thm}, devoted
to the actual proof of Theorem \ref{Th1.1}.
As tools for establishing the results in Sections~\ref{sec:pre} and \ref{sec:thm}
we utilize some congruences from \cite{mao,mp3} and
several combinatorial identities that can be found and proved by the
package \texttt{Sigma} \cite{S} via the software \texttt{Mathematica}.

 \section{Preliminary Lemmas}\label{sec:pre}
 \setcounter{lemma}{0}
\setcounter{theorem}{0}
\setcounter{corollary}{0}
\setcounter{remark}{0}
\setcounter{equation}{0}
\setcounter{conjecture}{0}
Recall that the Bernoulli polynomials $\{B_n(x)\}$ are given by
$$B_n(x)=\sum_{k=0}^n\binom nkB_kx^{n-k}\ \ (n=0,1,2,\ldots),$$
where, as before, $\{B_n\}$ are the Bernoulli numbers.
We will also use the classical Legendre symbol $\big(\frac aq\big)$
(for integer $a$ and odd prime $q$).
The following lemma involving the (generalized) harmonic numbers
can be easily deduced from \cite[Theorem~5.2~(c)]{s2000},
\cite[Theorem~3.9 (ii), (iii), (iv)]{s2008},
\cite[third equation on p.~302]{s2008}, and the simple
identity
$$ 
\sum_{1\le k<\frac{2p}3}\frac 1k=
\sum_{1\le k<\frac{2p}3}\frac{(-1)^{k-1}}k
+\sum_{1\le k<\frac p3}\frac 1k.
$$
\begin{lemma}\label{sunh} Let $p>5$ be a prime. Then
\begin{align*}
H_{\frac{p-1}2}&\equiv-2q_p(2)\;\pmod{p},\\
H_{\lfloor\frac{p}6\rfloor}&\equiv-2q_p(2)-\frac32q_p(3)\;\pmod{p},\\
H_{\lfloor\frac{p}3\rfloor}^{(2)}&\equiv\frac12\left(\frac{p}3\right)B_{p-2}\Big(\frac13\Big)\;\pmod p,\\
H_{\lfloor\frac{p}3\rfloor}&\equiv-\frac32q_p(3)+\frac{3p}4q^2_p(3)-\frac{p}6\left(\frac{p}3\right)B_{p-2}\Big(\frac13\Big)\;\pmod{p^2},\\
H_{\lfloor\frac{2p}3\rfloor}&\equiv-\frac32q_p(3)+\frac{3p}4q^2_p(3)+\frac{p}3\left(\frac{p}3\right)B_{p-2}\Big(\frac13\Big)\;\pmod{p^2}.
\end{align*}
\end{lemma}
\begin{lemma}\label{Lem2.2} Let $p>2$ be a prime. If $0\leq j\leq (p-1)/2$, then we have
$$
\binom{3j}j\binom{p+j}{3j+1}\equiv \frac{p}{3j+1}(1-pH_{2j}+pH_j)\;\pmod{p^3}.
$$
\end{lemma}
\begin{proof} If $0\leq j\leq (p-1)/2$ and $j\neq(p-1)/3$, then we have
\begin{align*}
\binom{3j}j\binom{p+j}{3j+1}&=\frac{(p+j)\cdots(p+1)p(p-1)\cdots(p-2j)}{j!(2j)!(3j+1)}\\
&\equiv\frac{pj!(1+pH_j)(-1)^{2j}(2j)!(1-pH_{2j})}{j!(2j)!(3j+1)}\\
&\equiv\frac{p}{3j+1}(1-pH_{2j}+pH_j)\;\pmod{p^3}.
\end{align*}
If $j=(p-1)/3$, then by Lemma \ref{sunh}, we have
\begin{align*}
&\binom{p-1}{\frac{p-1}3}\binom{p+\frac{p-1}3}{\frac{p-1}3}\\
&\equiv\left(1-pH_{\frac{p-1}3}+\frac{p^2}2(H_{\frac{p-1}3}^2-H_{\frac{p-1}3}^{(2)})\right)\left(1+pH_{\frac{p-1}3}+\frac{p^2}2(H_{\frac{p-1}3}^2-H_{\frac{p-1}3}^{(2)})\right)\\
&\equiv 1-p^2H_{\frac{p-1}3}^{(2)}\equiv1-\frac{p^2}2\left(\frac{p}3\right)B_{p-2}\Big(\frac13\Big)\;\pmod{p^3}
\end{align*}
and
$$
1-pH_{\frac{2p-2}3}+pH_{\frac{p-1}3}\equiv1-\frac{p^2}2\left(\frac{p}3\right)B_{p-2}\Big(\frac13\Big)\;\pmod{p^3}.
$$
This completes the proof of Lemma \ref{Lem2.2}.
\end{proof}
\begin{lemma}\label{mpt} Let $p>3$ be a prime. For any $p$-adic integer $t$, we have
\begin{equation*}
 \binom{\frac{2p-2}3+pt}{\frac{p-1}2}\equiv\binom{\frac{2p-2}3}{\frac{p-1}2}\left(1+pt(H_{\frac{2p-2}3}-H_{\frac{p-1}6})\right)\;\pmod{p^2}.
 \end{equation*}
 \end{lemma}
\begin{proof} Set $m=(2p-2)/3$. It is easy to check that
\begin{align*}
\binom{m+pt}{(p-1)/2}&=\frac{(m+pt)\cdots(m+pt-(p-1)/2+1)}{((p-1)/2)!}\\
&\equiv\frac{m\cdots(m-(p-1)/2+1)}{((p-1)/2)!}(1+pt(H_m-H_{m-(p-1)/2})\\
&=\binom{m}{(p-1)/2}(1+pt(H_m-H_{m-(p-1)/2})\;\pmod{p^2}.
\end{align*}
which completes the proof of Lemma \ref{mpt}.
\end{proof}
\begin{lemma}\label{3j+4p^3} Let $p>3$ be a prime. If $p=x^2+3y^2\equiv1\pmod3$, then
$$
p\sum_{k=0}^{\frac{p-1}2}\frac{\binom{2k}k^2}{(3k+4)16^k}\equiv\frac4{25}\left(4x^2-2p-\frac{p^2}{4x^2}\right)\;\pmod{p^3}.
$$
\end{lemma}
\begin{proof} By using \texttt{Sigma}, we establish the following identity:
$$
\sum_{k=0}^n\frac{\binom nk\binom{n+k}k(-1)^k}{3k+4}=-\frac{1}{(3n-1)(3n+1)(3n+4)}\prod_{k=1}^n\frac{3k-1}{3k-2}.
$$
(In terms of classical identities for hypergeometric series,
this evaluation is equivalent to the
$(a,b,c)\mapsto (n+1,4/3,1)$ case of the
Pfaff--Saalsch\"utz summation \cite[Appendix III, Equation (III.2)]{slater}.)
So modulo $p^3$, we have
\begin{align*}
&p\sum_{k=0}^{\frac{p-1}2}\frac{\binom{2k}k^2}{(3k+4)16^k}\\
&\equiv p\sum_{k=0}^{\frac{p-1}2}\frac{\binom{\frac{p-1}2}k\binom{\frac{p-1}2+k}{k}(-1)^k}{3k+4}+\binom{-\frac12}{\frac{p-4}3}^2+\binom{\frac{p-1}2}{\frac{p-4}3}\binom{\frac{p-1}2+\frac{p-4}3}{\frac{p-4}3}\\
&=\frac{4}{25-9p^2}\frac{(\frac23)_{\frac{p-1}2}}{(\frac13)_{\frac{p-1}3}(\frac{p}3+1)_{\frac{p-1}6}}+\binom{-\frac12}{\frac{p-4}3}^2+\binom{\frac{p-1}2}{\frac{p-4}3}\binom{\frac{p-1}2+\frac{p-4}3}{\frac{p-4}3},
\end{align*}
where we used the standard notation for the shifted factorial
$(a)_n=\prod_{j=0}^{n-1}(a+j)$ (cf.\ \cite[Section 1.1.1]{slater}).
It is easy to check that
\begin{align}
\binom{-\frac12}{\frac{p-4}3}^2&=\frac{4(p-1)^2}{(2p-5)^2}\binom{-\frac12}{\frac{p-1}3}^2,\label{-12p-43}\\
  \binom{\frac{p-1}2}{\frac{p-4}3}\binom{\frac{p-1}2+\frac{p-4}3}{\frac{p-4}3}
  &=\frac{4(p-1)}{5(p+5)}\binom{\frac{p-1}2}{\frac{p-1}3}\binom{\frac{p-1}2+\frac{p-1}3}{\frac{p-1}3}.\notag
\end{align}
These identities, together with \cite[pp.~14]{mao}, yield
\begin{align*}
  &p\sum_{k=0}^{\frac{p-1}2}\frac{\binom{2k}k^2}{(3k+4)16^k}\\
  &\equiv\frac{4}{25-9p^2}\frac{(\frac23)_{\frac{p-1}2}}
    {(\frac13)_{\frac{p-1}3}(\frac{p}3+1)_{\frac{p-1}6}}+\frac{4(p-1)^2}{(2p-5)^2}
    \binom{-\frac12}{\frac{p-1}3}^2\\
&\quad\;+\frac{4(p-1)}{5(p+5)}\binom{\frac{p-1}2}{\frac{p-1}3}\binom{\frac{p-1}2+\frac{p-1}3}{\frac{p-1}3}\\
  &\equiv\frac4{25}\left(1+\frac{9p^2}{25}\right)(-1)^{\frac{p-1}6}\binom{\frac{p-1}2}{\frac{p-1}3}\binom{\frac{2p-2}3}{\frac{p-1}2}\\
  &\quad\;\times
    \left(1-\frac{2p}3q_p(2)+\frac{5p^2}9q^2_p(2)
    +\frac{5p^2}{12}\left(\frac{p}3\right)B_{p-2}\Big(\frac13\Big)\right)\\
  &\quad\;+\frac4{25}\left(1-\frac{6p}5-\frac{3p^2}{25}\right)
    \binom{\frac{p-1}2}{\frac{p-1}3}^2\\
  &\qquad\;\;\times\left(1-\frac{3p}2q_p(3)+\frac{15p^2}8q^2_p(3)+\frac{5p^2}{24}\left(\frac{p}3\right)B_{p-2}\Big(\frac13\Big)\right)\\
&\quad\;-\frac4{25}\left(1-\frac{6p}5-\frac{6p^2}{25}\right)\binom{\frac{p-1}2}{\frac{p-1}3}\binom{\frac{5p-5}6}{\frac{p-1}3}\quad\pmod{p^3}.
\end{align*}
Again, by \cite[pp.~14--15]{mao}, we have
\begin{align*}
  &p\sum_{k=0}^{\frac{p-1}2}\frac{\binom{2k}k^2}{(3k+4)16^k}
  \\&\equiv\frac4{25}\left(1+\frac{9p^2}{25}\right)\left(4x^2-2p-\frac{p^2}{4x^2}\right)\\
  &\quad\;\times\left(1+\frac{2p}3q_p(2)-\frac{p^2}9q^2_p(2)
    +\frac{5p^2}{24}\left(\frac{p}3\right)B_{p-2}\Big(\frac13\Big)\right)\\
  &\quad\;\times\left(1-\frac{2p}3q_p(2)+\frac{5p^2}9q^2_p(2)
    +\frac{5p^2}{12}\left(\frac{p}3\right)B_{p-2}\Big(\frac13\Big)\right)\\
  &\quad\;+\frac4{25}\left(1-\frac{6p}5-\frac{3p^2}{25}\right)\left(4x^2-2p-\frac{p^2}{4x^2}\right)\\
  &\qquad\;\;\times\bigg(1-\frac{3p}2q_p(3)
    +\frac{15p^2}8q^2_p(3)+\frac{5p^2}{24}\left(\frac{p}3\right)B_{p-2}\Big(\frac13\Big)\bigg)\\
  &\qquad\;\;\times\bigg(1-\frac{4p}3q_p(2)+\frac{3p}2q_p(3)
    +\frac{14p^2}9q^2_p(2)-2p^2q_p(2)q_p(3)\\
  &\qquad\qquad+\frac{3p^2}8q^2_p(3)
    +\frac{p^2}{8}\left(\frac{p}3\right)B_{p-2}
    \Big(\frac13\Big)\bigg)\\
  &\quad\;-\frac4{25}\left(1-\frac{6p}5-\frac{6p^2}{25}\right)\left(4x^2-2p-\frac{p^2}{4x^2}\right)\\
  &\qquad\;\;\times
    \left(1-\frac{4p}3q_p(2)+\frac{14p^2}9q^2_p(2)
    +\frac{23p^2}{24}\left(\frac{p}3\right)B_{p-2}
    \Big(\frac13\Big)\right)\;\pmod{p^3}.
\end{align*}
It is easy to check that the right-side of the above congruence is congruent
to $\frac{4}{25}\left(4x^2-2p-\frac{p^2}{4x^2}\right)$ modulo $p^3$.
Therefore we immediately get the desired result stated in Lemma \ref{3j+4p^3}.
\end{proof}
\begin{lemma}\label{fuzhu} Let $p>3$ be a prime with $p=x^2+3y^2\equiv1\pmod3$ and let $k=(p-4)/3$. Then
\begin{align*}
  &\big(k(k+1)(k+3)+(2-k^2)(3k+1)p-(k+2)(3k+1)(3k+2)p^2\big)\\
  &\times\frac{\binom{-\frac12}{k}^2}{(k+1)(3k+2)}
    \left(\frac{\binom{3k}{k}\binom{p+k}{3k+1}}{3k+4}
    -\frac{1-pH_{2k}+pH_k}{3k+1}\right)\\
  &\equiv-\frac{184p^2x^2}{125}\;\pmod{p^3}
\end{align*}
and
\begin{align*}
  &\big(k(1+2k)+2p(k+1)(3k+1)-2p^2(3k+1)(3k+2)\big)\\
  &\times\frac{\binom{-\frac12}{k}^2}{3k+2}
    \left(\frac{\binom{3k}{k}\binom{p+2k}{3k+1}}{3k+4}
    +\frac{1+pH_{2k}-pH_k}{3k+1}\right)\\
  &\equiv-\frac{184p^2x^2}{125}\;\pmod{p^3}.
\end{align*}
\end{lemma}
\begin{proof} We only prove the first congruence; the proof
  of the second congruence is similar. It is easy to see that
\begin{align*}
&\frac{\binom{3k}{k}\binom{p+k}{3k+1}}{3k+4}-\frac{1-pH_{2k}+pH_k}{3k+1}\\
&=\binom{3k+3}{k+1}\binom{p+k+1}{3k+4}\frac{2(2p-5)(p-1)^2}{(4p-1)(p-3)(p+2)(p+5)}\\
&\quad\;-\frac{1-pH_{2k+2}+pH_{k+1}+\frac{p}{2k+2}+\frac{p}{2k+1}-\frac{p}{k+1}}{3k+1}.
\end{align*}
By Lemma \ref{Lem2.2} we have
$$
\binom{3k+3}{k+1}\binom{p+k+1}{3k+4}\equiv1-pH_{2k+2}+pH_{k+1}\pmod{p^3}.
$$
Thus,
$$
\frac{\binom{3k}{k}\binom{p+k}{3k+1}}{3k+4}-\frac{1-pH_{2k}+pH_k}{3k+1}\equiv-\frac{207p^2}{100}\;\pmod{p^3}.
$$
Together with \eqref{-12p-43} and
$\binom{(p-1)/2}{(p-1)/3}\equiv2x\pmod p$ (cf.\ \cite{YMK}), this yields
\begin{align*}
  &\big(k(k+1)(k+3)+(2-k^2)(3k+1)p-(k+2)(3k+1)(3k+2)p^2\big)\\
  &\times\frac{\binom{-\frac12}{k}^2}{(k+1)(3k+2)}
    \left(\frac{\binom{3k}{k}\binom{p+k}{3k+1}}{3k+4}
    -\frac{1-pH_{2k}+pH_k}{3k+1}\right)\\
  &\equiv-\frac{184p^2x^2}{125}\quad\pmod{p^3},
\end{align*}
which completes the proof of Lemma \ref{fuzhu}.
\end{proof}
\begin{lemma}\label{p2j} Let $p>2$ be a prime. If $0\leq j\leq (p-1)/2$, then we have
$$
\binom{3j}j\binom{p+2j}{3j+1}\equiv \frac{p(-1)^j}{3j+1}(1+pH_{2j}-pH_j)\;\pmod{p^3}.
$$
and
$$
1-pH_{\frac{2p-2}3}+pH_{\frac{p-1}3}\equiv1-\frac{p^2}2\left(\frac{p}3\right)B_{p-2}\Big(\frac13\Big)\;\pmod{p^3}.
$$
If $(p+1)/2\leq j\leq p-1$, then
$$
\binom{3j}j\binom{p+2j}{3j+1}\equiv \frac{2p(-1)^j}{3j+1}\;\pmod{p^2}.
$$
\end{lemma}
\begin{proof}If $0\leq j\leq (p-1)/2$ and $j\neq(p-1)/3$, then we have
\begin{align*}
\binom{3j}j\binom{p+2j}{3j+1}&=\frac{(p+2j)\cdots(p+1)p(p-1)\cdots(p-j)}{(3j+1)j!(2j)!}\\
&\equiv\frac{p(2j)!(1+pH_{2j})(-1)^{j}(j)!(1-pH_j)}{(3j+1)j!(2j)!}\\
&=\frac{p(-1)^j}{3j+1}(1+pH_{2j}-pH_{j})\quad\pmod{p^3}.
\end{align*}
If $j=(p-1)/3$, then by Lemma \ref{sunh} and $H_{p-1-k}^{(2)}\equiv-H_k^{(2)}\pmod p$, we have
\begin{align*}
&\binom{p-1}{\frac{p-1}3}\binom{p+\frac{2p-2}3}{\frac{2p-2}3}\\
&\equiv\left(1-pH_{\frac{p-1}3}+\frac{p^2}2(H_{\frac{p-1}3}^2-H_{\frac{p-1}3}^{(2)})\right)\left(1+pH_{\frac{2p-2}3}+\frac{p^2}2(H_{\frac{2p-2}3}^2-H_{\frac{2p-2}3}^{(2)})\right)\\
&\equiv 1+p(H_{\frac{2p-2}3}-H_{\frac{p-1}3})=\frac{p(-1)^j}{3j+1}(1+pH_{2j}-pH_{j})\qquad\pmod{p^3}.
\end{align*}
If $(p+1)/2\leq j\leq p-1$, then
\begin{align*}
&\binom{3j}j\binom{p+2j}{3j+1}\\
&=\frac{(p+2j)\cdots(2p+1)(2p)(2p-1)\cdots(p+1)p(p-1)\cdots(p-j)}{(3j+1)j!(2j)!}\\
&\equiv\frac{2p^2(2j)\cdots(p+1)(p-1)!(-1)^{j}(j)!}{(3j+1)j!(2j)!}\equiv\frac{2p(-1)^j}{3j+1}\;\pmod{p^2},
\end{align*}
which completes the proof of Lemma \ref{p2j}.
\end{proof}
\begin{lemma}\label{h2jhj3j+4} Let $p>3$ be a prime with $p=x^2+3y^2\equiv1\pmod3$. Then
\begin{align*}
p\sum_{j=0}^{\frac{p-1}2}\frac{\binom{2j}j^2(H_{2j}-H_j)}{(3j+4)16^j}\equiv-\frac{18}{125}(4x^2-2p)\;\pmod{p^2}.
\end{align*}
\end{lemma}
\begin{proof} By using \texttt{Sigma}, we establish the following identity:
\begin{align*}
&\sum_{j=0}^n\frac{\binom{n}j\binom{n+j}j(-1)^j(H_{2j}-H_j)}{3j+4}=-\frac{9(2n+1)}{10(3n-1)(3n+4)}\\
&+\frac{(\frac23)_n}{(3n-1)(3n+1)(3n+4)(\frac13)_n}\left(\frac9{10}+\sum_{k=1}^n\frac{(\frac13)_k}{k(\frac23)_k}\right).
\end{align*}
Substituting $n=(p-1)/2$ into the above identity, then modulo $p^2$ we have
\begin{align*}
p\sum_{j=0}^{\frac{p-1}2}\frac{\binom{2j}j^2(H_{2j}-H_j)}{(3j+4)16^j}\equiv \frac{p(\frac23)_{\frac{p-1}2}}{\frac{3p-5}2\frac{3p-1}2\frac{3p+5}2(\frac13)_{\frac{p-1}2}}\left(\frac9{10}+\sum_{k=1}^{\frac{p-1}2}\frac{(\frac13)_k}{k(\frac23)_k}\right).
\end{align*}
In view of \cite[pp.\ 9]{mp3} and \cite[pp.\ 14--15]{mao}, we have
\begin{align*}
  \frac{(\frac23)_{\frac{p-1}2}}
  {(\frac13)_{\frac{p-1}3}(\frac{p}3+1)_{\frac{p-1}6}}
  &\equiv 4x^2-2p\quad\pmod{p^2},\\
  \frac{(\frac23)_{\frac{p-1}2}}{(\frac13)_{\frac{p-1}2}}
  \sum_{k=1}^{\frac{p-1}2}\frac{(\frac13)_k}{k(\frac23)_k}&\equiv 0\quad\pmod{p}.
\end{align*}
Hence,
\begin{align*}
p\sum_{j=0}^{\frac{p-1}2}\frac{\binom{2j}j^2(H_{2j}-H_j)}{(3j+4)16^j}&\equiv \frac9{10}\frac{p(\frac23)_{\frac{p-1}2}}{\frac{3p-5}2\frac{3p-1}2\frac{3p+5}2(\frac13)_{\frac{p-1}2}}\\
\equiv-\frac9{10}\frac{4}{25}(4x^2-2p)&=-\frac{18}{125}(4x^2-2p)\quad\pmod{p^2}.
\end{align*}
The proof of Lemma \ref{h2jhj3j+4} is complete.
\end{proof}
\section{Proof of Theorem \ref{Th1.1}}\label{sec:thm}
\setcounter{lemma}{0}
\setcounter{theorem}{0}
\setcounter{corollary}{0}
\setcounter{remark}{0}
\setcounter{equation}{0}
\setcounter{conjecture}{0}
Our proof of Theorem \ref{Th1.1} heavily relies on the following
two transformation formulas due to H.-H.~Chan and W.~Zudilin \cite{CZ}
and Z.-H.~Sun \cite{sund} respectively,
 \begin{align}\label{2.0}\sum_{k=0}^n\binom{n}k^2\binom{2k}k\binom{2n-2k}{n-k}=\sum_{k=0}^{n}(-1)^k\binom{n+2k}{3k}\binom{2k}k^2\binom{3k}k16^{n-k},
 \end{align}
 \begin{align}\label{2.1}\sum_{k=0}^n\binom{n}k^2\binom{2k}k\binom{2n-2k}{n-k}=\sum_{k=0}^{\lfloor n/2\rfloor}\binom{n+k}{3k}\binom{2k}k^2\binom{3k}k4^{n-2k}.
 \end{align}
 \noindent{\it Proof of Theorem \ref{Th1.1}}. We first consider
 the first congruence in Theorem \ref{Th1.1} in the case
 $p=x^2+3y^2\equiv1\pmod3$. By \eqref{2.1}, we have
\begin{align*}
\sum_{k=0}^{p-1}k^3\frac{D_k}{4^k}&=\sum_{k=0}^{p-1}\frac{k^3}{4^k}\sum_{j=0}^{\lfloor k/2\rfloor}\binom{k+j}{3j}\binom{2j}j^2\binom{3j}j4^{k-2j}\notag\\
&=\sum_{j=0}^{(p-1)/2}\frac{\binom{2j}j^2\binom{3j}j}{16^j}\sum_{k=2j}^{p-1}k^3\binom{k+j}{3j}.
\end{align*}
 By using \texttt{Sigma}, we establish the following identity:
\begin{align*}
\sum_{k=2j}^{n-1}k^3\binom{k+j}{3j}=\frac{\Sigma_1}{(j+1)(3j+2)(3j+4)}\binom{n+j}{3j+1},
\end{align*}
where
\begin{align*}
\Sigma_1&=j(j+1)(j+3)+n(2-j^2)(3j+1)-n^2(j+2)(3j+1)(3j+2)\\
&\quad\;+n^3(j+1)(3j+1)(3j+2).
\end{align*}
Thus,
\begin{align*}
  \sum_{k=0}^{p-1}k^3\frac{D_k}{4^k}
  =\sum_{j=0}^{\frac{p-1}2}\frac{\binom{2k}k^2\binom{3j}j
  \binom{p+j}{3j+1}}{16^k}\frac{\Sigma_1}{(j+1)(3j+2)(3j+4)}.
\end{align*}
Let
$$
\Sigma_2=k(k+1)(k+3)+p(2-k^2)(3k+1)-p^2(k+2)(3k+1)(3k+2).
$$
In view of Lemma \ref{Lem2.2}, we have for $p\equiv1\pmod3$ the supercongruence
\begin{align*}
&\sum_{k=0}^{p-1}k^3\frac{D_k}{4^k}\notag\equiv\sum_{k=0}^{\frac{p-1}2}\frac{\binom{2k}k^2\binom{3k}k\binom{p+k}{3k+1}}{16^k}\frac{\Sigma_2}{(k+1)(3k+2)(3k+4)}\\
&\equiv\sum_{k=0}^{\frac{p-1}2}\frac{\binom{2k}k^2}{16^k}\frac{p(1-pH_{2k}+pH_k)\Sigma_2}{(k+1)(3k+1)(3k+2)(3k+4)}+S_1\quad\pmod{p^3},
\end{align*}
where $S_1$ is defined by the following expression with $k=(p-4)/3$,
\begin{align*}
S_1&=\big(k(k+1)(k+3)+(2-k^2)(3k+1)p-(k+2)(3k+1)(3k+2)p^2\big)\\
   &\quad\;\times\frac{\binom{-\frac12}{k}^2}
     {(k+1)(3k+2)}\left(\frac{\binom{3k}{k}\binom{p+k}{3k+1}}{3k+4}
-\frac{1-pH_{2k}+pH_k}{3k+1}\right).
\end{align*}
In view of \cite[Equation~(3.5)]{s2018} and \cite{mp3}, we have
\begin{gather}
  \sum_{k=0}^{\frac{p-1}2}\frac{\binom{2k}k^2}{(k+1)16^k}\equiv0\pmod{p^2},\quad
  \sum_{k=0}^{\frac{p-1}2}\frac{\binom{2k}k^2(H_{2k}-H_k)}{(3k+1)16^k}
  \equiv0\pmod p,\label{zhongyao1}\\
  \frac23\sum_{k=0}^{\frac{p-1}2}\frac{\binom{2k}k^2(H_{2k}-H_k)}{(3k+2)16^k}
  \equiv\sum_{k=0}^{\frac{p-1}2}\frac{p\binom{2k}k^2}{(3k+2)16^k}\equiv-\frac{p^2}{x^2}\pmod{p^3}\label{zhongyao2}.
\end{gather}
Hence by Lemma \ref{fuzhu} and \cite[Theorem 1.2]{mao}, we have
\begin{align*}
&\sum_{k=0}^{p-1}k^3\frac{D_k}{4^k}+\frac{184p^2x^2}{125}\equiv\frac{p}{27}\sum_{k=0}^{\frac{p-1}2}\frac{\binom{2k}k^2}{16^k}\left(\frac{-8}{3k+1}+\frac{21}{3k+2}-\frac{10}{3k+4}\right)\\
&\quad\;+\frac{p^2}3\sum_{k=0}^{\frac{p-1}2}\frac{\binom{2k}k^2}{16^k}\left(\frac{-3}{k+1}+\frac{7}{3k+2}+\frac{1}{3k+4}\right)\\
&\quad\;-p^3\sum_{k=0}^{\frac{p-1}2}\frac{\binom{2k}k^2}{16^k}\left(\frac{1}{k+1}-\frac{2}{3k+4}\right)\\
&\quad\;-\frac{p^2}{27}\sum_{k=0}^{\frac{p-1}2}\frac{\binom{2k}k^2(H_{2k}-H_k)}{16^k}\left(\frac{-8}{3k+1}+\frac{21}{3k+2}-\frac{10}{3k+4}\right)\\
&\quad\;-\frac{p^3}3\sum_{k=0}^{\frac{p-1}2}\frac{\binom{2k}k^2(H_{2k}-H_k)}{16^k}\left(\frac{-3}{k+1}+\frac{7}{3k+2}+\frac{1}{3k+4}\right)\\
&\equiv\left(-\frac8{27}-\frac{10}{27}\frac4{25}\right)\left(4x^2-2p-\frac{p^2}{4x^2}\right)-\frac{21}{27}\frac{p^2}{x^2}+\frac{p}3\frac{4}{25}(4x^2-2p)\\
&\quad\;+2p^2\frac{16x^2}{25}-\frac{10p}{27}\frac{18}{125}(4x^2-2p)+\frac{21}{27}\frac32\frac{p^2}{x^2}+\frac{p^2}3\frac{18}{125}4x^2\\
&\equiv-\frac{64x^2}{45}+\frac{32p}{45}+\frac{43p^2}{90x^2}+\frac{184p^2x^2}{125}\quad\pmod{p^3}.
\end{align*}
Thus we immediately obtain the desired result
\begin{equation}\label{congruence1}
\sum_{k=0}^{p-1}k^3\frac{D_k}{4^k}\equiv-\frac{64x^2}{45}+\frac{32p}{45}+\frac{43p^2}{90x^2}\quad\pmod{p^3}.
\end{equation}
Now we are ready to prove the case $p\equiv2\pmod3$ with $p>5$.
Similar to before,
\begin{align*}
&\sum_{k=0}^{p-1}k^3\frac{D_k}{4^k}\notag\equiv\sum_{k=0}^{\frac{p-1}2}\frac{\binom{2k}k^2\binom{3k}k\binom{p+k}{3k+1}}{16^k}\frac{k(k+1)(k+3)+p(2-k^2)(3k+1)}{(k+1)(3k+2)(3k+4)}\\
&\equiv\sum_{k=0}^{\frac{p-1}2}\frac{\binom{2k}k^2}{16^k}\frac{p(1-pH_{2k}+pH_k)(k(k+1)(k+3)+p(2-k^2)(3k+1))}{(k+1)(3k+1)(3k+2)(3k+4)}\\
&\equiv\frac{p}{27}\sum_{k=0}^{\frac{p-1}2}\frac{\binom{2k}k^2}{16^k}\left(\frac{-8}{3k+1}+\frac{21}{3k+2}-\frac{10}{3k+4}\right)\\
&\quad\;+\frac{p^2}3\sum_{k=0}^{\frac{p-1}2}\frac{\binom{2k}k^2}{16^k}\left(\frac{-3}{k+1}+\frac{7}{3k+2}+\frac{1}{3k+4}\right)\\
&\quad\;-\frac{p^2}{27}\sum_{k=0}^{\frac{p-1}2}\frac{\binom{2k}k^2(H_{2k}-H_k)}{16^k}\left(\frac{-8}{3k+1}+\frac{21}{3k+2}-\frac{10}{3k+4}\right)\;\pmod{p^2}.
\end{align*}
In view of \cite{mp3}, we have
$$\sum_{k=0}^{\frac{p-1}2}\frac{\binom{2k}k^2}{(3k+1)16^k}\equiv0\quad\pmod p,$$
and in view of Lemma \ref{3j+4p^3}, we have
$$
\sum_{k=0}^{\frac{p-1}2}\frac{\binom{2k}k^2}{(3k+4)16^k}\equiv0\quad\pmod p.
$$
Thus,
\begin{align*}
\sum_{k=0}^{p-1}k^3\frac{D_k}{4^k}&\equiv\frac{7p}{9}\sum_{k=0}^{\frac{p-1}2}\frac{\binom{2k}k^2}{(3k+2)16^k}+\frac{7p^2}3\sum_{k=0}^{\frac{p-1}2}\frac{\binom{2k}k^2}{(3k+2)16^k}\\
&\quad\;-\frac{7p^2}{9}\sum_{k=0}^{\frac{p-1}2}\frac{\binom{2k}k^2(H_{2k}-H_k)}{(3k+2)16^k}\quad\pmod{p^2}.
\end{align*}
In view of \cite[Equation~(4.2)]{mp3}, we have
\begin{align}\label{yaoyong1}
\sum_{k=0}^{\frac{p-1}2}\frac{\binom{2k}k^2(H_{2k}-H_k)}{16^k(3k+2)}\equiv3\sum_{k=0}^{\frac{p-1}2}\frac{\binom{2k}k^2}{16^k(3k+2)}\quad\pmod p,
\end{align}
and
\begin{equation}\label{yaoyong2}
p\sum_{k=0}^{\frac{p-1}2}\frac{\binom{2k}k^2}{(3k+2)16^k}\equiv4R_3(p)\quad\pmod{p^2}.
\end{equation}
Hence
\begin{equation}\label{congruence31}
\sum_{k=0}^{p-1}k^3\frac{D_k}{4^k}\equiv\frac{28}9R_3(p)\quad\pmod{p^2}.
\end{equation}

\noindent Now we consider the other congruences in Theorem \ref{Th1.1}.
Similar to above, by \eqref{2.0}, we have
\begin{align*}
\sum_{k=0}^{p-1}k^3\frac{D_k}{16^k}&=\sum_{k=0}^{p-1}\frac{k^3}{16^k}\sum_{j=0}^{k}(-1)^j\binom{k+2j}{3j}\binom{2j}j^2\binom{3j}j16^{k-j}\notag\\
&=\sum_{j=0}^{p-1}\frac{\binom{2j}j^2\binom{3j}j}{(-16)^j}\sum_{k=j}^{p-1}k^3\binom{k+2j}{3j}.
\end{align*}
By using \texttt{Sigma}, we establish the following identity:
\begin{align*}
\sum_{k=j}^{n-1}k^3\binom{k+2j}{3j}=\frac{\Sigma_3}{(3j+2)(3j+4)}\binom{n+2j}{3j+1},
\end{align*}
where
$$
\Sigma_3=j(1+2j)+2n(j+1)(3j+1)-2n^2(3j+1)(3j+2)+n^3(3j+1)(3j+2).
$$
Let
$$
\Sigma_4=j(1+2j)+2p(j+1)(3j+1)-2p^2(3j+1)(3j+2).
$$
Thus, if $p\equiv1\pmod3$, then modulo $p^3$, we have
\begin{align*}
&\sum_{k=0}^{p-1}k^3\frac{D_k}{16^k}-\frac1{18p(p+1)}\binom{-1/2}{\frac{2p-2}3}^2\binom{2p-2}{\frac{2p-2}3}\binom{p+\frac{4p-4}3}{2p-1}\\
&\equiv\sum_{j=0}^{\frac{p-1}2}\frac{\binom{2j}j^2\binom{3j}j\binom{p+2j}{3j+1}}{(-16)^j}\frac{\Sigma_4}{(3j+2)(3j+4)}\\
  &\equiv\sum_{j=0}^{\frac{p-1}2}\frac{\binom{2j}j^2p(-1)^j
    (1+pH_{2j}-pH_j)}{(-16)^j}\frac{\Sigma_4}{(3j+1)(3j+2)(3j+4)}+S_2,
\end{align*}
where $S_2$ is defined by the following expression with $k=(p-4)/3$,
\begin{align*}
&S_2=-\frac{\binom{-\frac12}{k}^2(k(1+2k)+2p(k+1)(3k+1)-2p^2(3k+1)(3k+2))}{3k+2}\\
&\quad\;\times\left(\frac{\binom{3k}{k}\binom{p+2k}{3k+1}}{3k+4}+\frac{1+pH_{2k}-pH_k}{3k+1}\right).
\end{align*}
Hence, similar to above, by \eqref{zhongyao1}, \eqref{zhongyao2}
and Lemma \ref{fuzhu}, we have
\begin{align*}
&\sum_{k=0}^{p-1}k^3\frac{D_k}{16^k}-\frac1{18p(p+1)}\binom{-1/2}{\frac{2p-2}3}^2\binom{2p-2}{\frac{2p-2}3}\binom{p+\frac{4p-4}3}{2p-1}\\
&\equiv\frac{p}{27}\sum_{j=0}^{\frac{p-1}2}\frac{\binom{2j}j^2}{16^j}\left(\frac{-1}{3j+1}+\frac{-3}{3j+2}+\frac{10}{3j+4}\right)\\
&\quad\;+\frac{p^2}3\sum_{j=0}^{\frac{p-1}2}\frac{\binom{2j}j^2}{16^j}\left(\frac{1}{3j+2}+\frac1{3j+4}\right)-2p^3\sum_{j=0}^{\frac{p-1}2}\frac{\binom{2j}j^2}{(3j+4)16^j}\\
&\quad\;+\frac{p^2}{27}\sum_{j=0}^{\frac{p-1}2}\frac{\binom{2j}j^2}{16^j}\left(\frac{-1}{3j+1}+\frac{-3}{3j+2}+\frac{10}{3j+4}\right)(H_{2j}-H_j)\\
&\quad\;+\frac{p^3}3\sum_{j=0}^{\frac{p-1}2}\frac{\binom{2j}j^2}{16^j}\left(\frac{1}{3j+2}+\frac1{3j+4}\right)(H_{2j}-H_j)+\frac{184p^2x^2}{125}\\
&\equiv\left(-\frac1{27}+\frac{10}{27}\frac4{25}\right)\left(4x^2-2p-\frac{p^2}{4x^2}\right)+\frac{1}{9}\frac{p^2}{x^2}+\frac{p}3\frac{4}{25}(4x^2-2p)\\
&\quad\;-2p^2\frac{16x^2}{25}-\frac{10p}{27}\frac{18}{125}(4x^2-2p)+\frac{1}{9}\frac32\frac{p^2}{x^2}-\frac{p^2}3\frac{18}{125}4x^2+\frac{184p^2x^2}{125}\\
&\equiv\frac{4x^2}{45}-\frac{2p}{45}+\frac{49p^2}{180x^2}\quad\pmod{p^3}.
\end{align*}
It is easy to see that
$$
\binom{2p-2}{\frac{2p-2}3}\binom{p+\frac{4p-4}3}{2p-1}\equiv-2p\quad\pmod{p^2}.
$$
In view of \cite[pp. 18]{mp3}, we have
$$\binom{-1/2}{\frac{2p-2}3}^2\equiv\frac{9p^2}{4x^2}\quad\pmod{p^3}.$$
These yield
\begin{align}\label{congruence3}
\sum_{k=0}^{p-1}k^3\frac{D_k}{16^k}&\equiv\frac{4x^2}{45}-\frac{2p}{45}+\frac{49p^2}{180x^2}-\frac{p^2}{4x^2}\notag\\
&=\frac{4x^2}{45}-\frac{2p}{45}+\frac{p^2}{45x^2}\quad\pmod{p^3}.
\end{align}
If $p\equiv2\pmod3$ with $p>5$ (the case $p=5$ can be checked directly), then modulo $p^2$, we have
\begin{align*}
&\sum_{k=0}^{p-1}k^3\frac{D_k}{16^k}\\
&\equiv\sum_{j=0}^{\frac{p-1}2}\frac{\binom{2j}j^2}{16^j}\frac{pj(1+2j)+2p^2(j+1)(3j+1)+p^2j(2j+1)(H_{2j}-H_j)}{(3j+1)(3j+2)(3j+4)}.
\end{align*}
Hence, similar to above, we have
\begin{align*}
\sum_{k=0}^{p-1}k^3\frac{D_k}{16^k}&\equiv\frac{p}{27}\sum_{j=0}^{\frac{p-1}2}\frac{\binom{2j}j^2}{16^j}\left(\frac{-1}{3j+1}+\frac{-3}{3j+2}+\frac{10}{3j+4}\right)\\
&\quad\;+\frac{p^2}3\sum_{j=0}^{\frac{p-1}2}\frac{\binom{2j}j^2}{16^j}\left(\frac{1}{3j+2}+\frac1{3j+4}\right)\\
&\quad\;+\frac{p^2}{27}\sum_{j=0}^{\frac{p-1}2}\frac{\binom{2j}j^2}{16^j}\left(\frac{-1}{3j+1}+\frac{-3}{3j+2}+\frac{10}{3j+4}\right)(H_{2j}-H_j)\\
&\equiv-\frac{p}9\sum_{j=0}^{\frac{p-1}2}\frac{\binom{2j}j^2}{(3j+2)16^j}+\frac{p^2}3\sum_{j=0}^{\frac{p-1}2}\frac{\binom{2j}j^2}{(3j+2)16^j}\\
&\quad\;-\frac{p^2}9\sum_{j=0}^{\frac{p-1}2}\frac{\binom{2j}j^2}{(3j+2)16^j}(H_{2j}-H_j)\\
&\equiv-\frac{p}9\sum_{j=0}^{\frac{p-1}2}\frac{\binom{2j}j^2}{(3j+2)16^j}\equiv-\frac49R_3(p)\quad\pmod{p^2}.
\end{align*}
This, together with \eqref{congruence1}, \eqref{congruence31} and \eqref{congruence3},
completes the proof of Theorem \ref{Th1.1}.\qed


\end{document}